\documentclass[12pt]{amsart}
\usepackage[latin1]{inputenc}
\usepackage{mathptmx}
\newtheorem{theorem}{Theorem}[section]
\newtheorem{lemma}[theorem]{Lemma}
\newtheorem{corollary}[theorem]{Corollary}
\newtheorem{proposition}[theorem]{Proposition}

\theoremstyle{definition}
\newtheorem{remark}[theorem]{Remark}
\newtheorem{definition}[theorem]{Definition}
\newtheorem{example}[theorem]{Example}
\newtheorem{remark/example}[theorem]{Remark/Example}

\let\oldlabel=\label
\def\prellabel{\marginparsep=1em\marginparwidth=44pt
 \def\label##1{\oldlabel{##1}\ifmmode\else\ifinner\else
 \marginpar{{\footnotesize\ \\ \tt
 ##1}}\fi\fi}}

\numberwithin{equation}{section}

\def\PP{ {\bf P} }
\def\NN{ {\bf N} }
\def\YY{ {\bf Y} }
\def\ZZ{ {\bf Z} }

\newcommand{\Rees}{\operatorname{Rees}}
\newcommand{\ini}{\operatorname{in}}

\newcommand{\GL}{\operatorname{GL}}

\newcommand{\mm}{\operatorname{{\mathbf m}}}

\newcommand{\depth}{\operatorname{depth}}
\newcommand{\height}{\operatorname{height}}
\newcommand{\grade}{\operatorname{grade}}

\newcommand{\Tor}{\operatorname{Tor}}

\newcommand{\projdim}{\operatorname{projdim}}
\newcommand{\reg}{\operatorname{reg}}

\newcommand{\Ass}{\operatorname{Ass}}

\newcommand{\Min}{\operatorname{Min}}

\numberwithin{equation}{section}

\begin{document}
\title{Maximal minors and linear powers}
\author{Winfried Bruns}
\address{Universit\"at Osnabr\"uck, Institut f\"ur Mathematik, 49069 Osnabr\"uck, Germany}
\email{wbruns@uos.de}
\author{Aldo Conca}
\address{ Dipartimento di Matematica,
Universit\`a degli Studi di Genova, Italy} \email{conca@dima.unige.it}
\author{Matteo Varbaro}
\address{Dipartimento di Matematica,
Universit\`a degli Studi di Genova, Italy}
\email{varbaro@dima.unige.it}
\date{}
\begin{abstract}
An ideal $I$ in a polynomial ring $S$ has linear powers if all the powers $I^k$ of $I$ have a linear free resolution. 
We show that the ideal of maximal minors of a sufficiently general matrix with linear entries has linear powers.
The required genericity is expressed in terms of the heights of the ideals of lower order minors. In particular we prove that every rational normal scroll has linear powers.
 \end{abstract}
\maketitle

\section{Introduction}
When $I$ is a homogeneous ideal in a polynomial ring $S$, it is known from work of Cutkosky, Herzog, Trung \cite{CHT} and Kodiyalam \cite{K} that the Castelnuovo-Mumford regularity of $I^k$ is asymptotically a linear function in $k$. Many authors have studied the function $\reg(I^k)$ from various perspectives, see for instance the recent papers of Eisenbud and Harris \cite{EHa1} and Chardin \cite{Ch}. This function behaves in the simplest possible way when $I$ is generated by forms of a given degree, say $d$, and all its powers have a linear resolution, i.e.~$\reg (I^k)=dk$ for all $k$. We term ideals with this property \emph{ideals with linear powers}. Similarly we say that a projective variety has linear powers when its defining ideal has linear powers.

The rational normal scrolls are important projective varieties. They have both a toric and a determinantal presentation and play a prominent role in the Bertini-Del Pezzo classification theorem of irreducible varieties of minimal degree, see the ``centennial account" of Eisenbud and Harris \cite{EHa} for details. A rational normal scroll of dimension $n$ is uniquely determined by a sequence of positive integers $a_1,\dots,a_n$, see \cite{Ha}. It is balanced if $|a_j-a_i| \leq 1$ for all $i,j$. In \cite{CHV} Conca, Herzog and Valla showed that the Rees algebra of a balanced rational normal scroll is defined by a Gröbner basis of quadrics and hence it is a Koszul algebra. It follows then from a result of Blum, see \cite{B} or Proposition \ref{rees}, that balanced rational normal scrolls have linear powers. Herzog, Hibi and Ohsugi ask in \cite{HHO} whether the same is true for every rational normal scroll.

One can ask the same question for irreducible varieties of minimal degree. However, apart from the rational normal scrolls, they are either quadric hypersurfaces (for which the question is trivial) or the Veronese embedding $\PP^2\to \PP^5$. The latter can be treated by ad hoc methods (see \cite{HHO} or Proposition \ref{Ver}). So the question is really open only for the rational normal scrolls. For them one could try to prove that the associated Rees algebra is Koszul (for example by exhibiting a Gröbner basis of quadrics for their defining ideals) but, despite to many efforts, there was no progress in this direction.

The rational normal scrolls are determinantal; they are defined by the $2$-minors of a $2\times n$ matrix with linear entries and expected height. The main result of the paper asserts that the ideal maximal minors of a $m\times n$ matrix $X$ of linear forms has linear powers provided the ideals of minors satisfy the following inequalities:
\begin{itemize}
\item[(1)] $\height I_m(X) \geq n-m+1,$
\item[(2)] $\height I_j(X)\geq \min\{ (m+1-j)(n-m)+1,N \}$ for every $j=2,\dots,\allowbreak m-1$,
\end{itemize}
where $N=\height I_1(X)$, see Theorem \ref{mainThm}. We prove also that, under the above height conditions, the Rees algebra of $I_m(X)$ is of fiber type and the Betti numbers of $I_m(X)^k$ depend only on the size of the matrix, the exponent $k$ and $N$. As a corollary, we answer the question of Herzog, Hibi and Ohsugi positively.
In Section \ref{xReg} we introduce the main notions and technical terms. We give a characterization of ideals with linear powers in terms of the $(1,0)$-regularity of their Rees algbera, see Theorem \ref{palombo}.

The computations that led to the discovery of the theorems and examples presented in this paper have been performed with CoCoA\cite{Cocoa}. Our work was partly supported by the 2011-12 Vigoni project ``Commutative algebra and combinatorics".

\section{Ideals with linear powers and their Rees algebras}
\label{xReg}

Let $S$ be the polynomial ring $K[x_1,\dots, x_n]$ over an infinite field $K$ with maximal homogeneous ideal $\mm_S=(x_1,\dots,x_n)$ , and let $I\subset S$ be a homogeneous ideal. The Castelnuovo-Mumford regularity of a graded $S$-module $M$ is denoted by $\reg(M)$.

\begin{definition} We say that $I$ has linear powers if all the powers of $I$ have a linear resolution. In other words, all the generators of $I$ have the same degree, say $d$, and $\reg(I^k)=kd$ for every $k\in \NN$.
\end{definition}
Examples of ideals with linear powers are 
\begin{itemize}
\item[(i)]  strongly stable monomial ideals generated in one degree,
\item[(ii)] products of ideals of linear forms, 
\item[(iii)] polymatroidal ideals, 
\item[(iv)] ideals with a linear resolution and dimension $\leq 1$,
\end{itemize} 
see Conca and Herzog \cite{CH}. Also, in \cite{HHZ} Herzog, Hibi and Zheng proved that monomial ideals generated in degree $2$ have linear powers as soon as they have a linear resolution.
In general however ideals with a linear resolution need not have linear powers; see Conca \cite{C} for a list of examples.

Properties of the powers of an ideal can often be expressed in terms of the Rees algebra of the ideal itself. The goal of this section is the discussion of properties of the Rees algebra of $I$ that are related to $I$ having linear powers.

Suppose that $I$ is generated by elements of degree $d$. The Rees algebra
$$
\Rees(I)=\bigoplus_{j\in \ZZ}I^j
$$
has a $\ZZ^2$-graded structure with
$$
\Rees(I)_{(i,j)}=(I^j)_{jd+i}.
$$
It is standard graded in the sense that the $K$-algebra generators of $\Rees(I)$ live in degree $(1,0)$ (the elements of $S_1$) and in degree $(0,1)$ (the elements of $I_d$). Let $m$ be the minimal number of generators of $I$ and consider the polynomial ring $A=S[y_1,\dots,y_m]=K[x_1,\dots, x_n,y_1,\dots,y_m]$, bigraded by setting $\deg x_i=(1,0)$ and $\deg y_j=(0,1)$ for $i=1,\dots,n$ and $j=1,\dots,m$. We have a $\ZZ^2$-graded presentation
$$
A/P(I)\cong \Rees(I).
$$

Let $Q(I)$ be the subideal of $P(I)$ generated by elements of degree $(*,1)$. By construction, $Q(I)$ defines the symmetric algebra $S(I)$ of $I$. The ideal $I$ is said to be of \emph{ linear type} if $P(I)=Q(I)$, i.e. if $\Rees(I)=S(I)$.
The fiber cone $F(I)$ of $I$ is, by definition,
$$
F(I)=\Rees(I)/\mm_S\Rees(I) \cong K[I_d]\subset S
$$
and can be presented as
$$
F(I)\cong K[y_1,\dots,y_m]/T(I)
$$
where $T(I)=P(I)\cap K[y_1,\dots,y_m]$. The ideal $I$ is said to be of \emph{ fiber type} if $P(I)=Q(I)+T(I)$.  Here we are identifying $T(I)$ with its extension to $A$.  In other words, $I$ is of fiber type if the ideal $P(I)$ has no minimal generators of degree $(a,b)$ with $a>0$ and $b>1$.

For a bigraded $A$-module $M$ let $\beta_{i,(a,b)}^A(M)$ denote the Betti number of $M$ corresponding to homological position $i$ and degree $(a,b)$. The $(1,0)$-regularity (also called $x$-regularity in \cite{ACD,HHO,R}) of $M$ is, by definition,
$$
\reg_{(1,0)}(M)=\sup \{ a-i : \beta_{i,(a,b)}^A(M)\neq 0\mbox{ for some } b\}.
$$
Furthermore $P_A(M)$ denotes the multigraded  Poincaré series of $M$, that is,
$$
P_A(M)=\sum \beta_{i,(a,b)}^A(M) x^i z^as^b
$$
We have the following result of Blum \cite[Cor.3.6]{B}:

\begin{proposition}
\label{rees}
If $\Rees(I)$ is Koszul, then $I$ has linear powers.
\end{proposition}

The converse of Proposition \ref{rees} does not hold in general, see Examples \ref{ex1}, \ref{ex2} and \ref{ex3} below.
We will make use of the following well-known fact, see for instance \cite[Prop.1.2]{CH}:

\begin{lemma}
\label{reg}
Let $M$ be a finitely generated graded $S$-module and $x\in S_1$ such that $0:_Mx$ has finite length. Set $a_0=\sup\{ j : H_{\mm}^0(M)_j\neq 0\}$. Then
$\reg(M)=\max \{ \reg(M/xM), a_0\}$.
\end{lemma}

Let $z\in S$ and set $J=(I+(z))/(z)\subset R=S/(z)$. We have a short exact sequence:
$$0\to W\to \Rees(I)/z\Rees(I)\to \Rees(J)\to 0 \eqno{(1)}$$
where
$$W=\bigoplus_{k\in \NN} (I^k\cap (z))/zI^k=\bigoplus_{k\in \NN} z(I^k:z)/zI^k.$$

We have:

\begin{lemma}
\label{palo}
 Assume that $I$ has linear powers. Let $z$ be a linear form in $S$ such that $(I^k:z)/I^k$ has finite length for all $k$ (a general linear form has this property). Set $J=(I+(z))/(z)\subset R=S/(z)$. Then:
\begin{itemize}
\item[(1)] $J$ has linear powers and the Betti numbers (over $R$) of the powers of $J$ are determined by those (over $S$) of the powers of $I$.
\item[(2)] If $I$ is of fiber type, then $J$ is of fiber type.
\item[(3)] The $(1,0)$-regularity of $\Rees(I)$ is $0$.
\end{itemize}
\end{lemma}

\begin{proof} (1) After a change of coordinates we may assume that $z=x_n$, and $R=K[x_1,\dots,x_{n-1}]$.
Denote the maximal homogeneous ideal of $R$ by $\mm_R$.
By Lemma \ref{reg} we have
$$
dk-1=\reg(S/I^k)=\max\{ \reg(R/J^k), a_0(S/I^k)\}.
$$
It follows that $J$ has linear powers. Also $a_0(S/I^k)=kd-1$ or $H^0_{\mm_S}(S/I^k)=0$. We also have $I^k:z=(V_k)+I^k$ where
$$
V_k=(I^k:z)_{kd-1}=(I^k:\mm_S)_{kd-1}=H^0_{\mm_S}(S/I^k)_{kd-1}
$$
and $\mm_SV_k\subset I^k$. The Betti numbers of an ideal with linear resolution are determined by its Hilbert function. Comparing Hilbert functions and taking into account that $\dim V_k$ is $\beta^S_{n-1}(I^k)$ it follows that
$$\beta^R_i(J^k)=\beta^S_i(I^k)-\binom{n-1}{i}\beta^S_{n-1}(I^k)$$
and hence the Betti numbers of the powers of $J$ are determined by those of the powers of $I$.

(2) We have a presentation $\Rees(I)/x_n\Rees(I)=R[y_1,\dots,y_m]/P_1$ where $P_1=(P(I)+(x_n))/(x_n)$ and we may represent the ideal $W$ as $(U+P_1)/P_1$ with $U$ generated in degrees $(0,*)$ and $(x_1,\dots,x_{n-1})U\subset P_1$. By construction, $\Rees(J)\cong R[y_1,\dots,y_m]/(U+P_1)$. Note that this might not be the minimal presentation of $\Rees(J)$ because the number of generators of $J$ can be smaller than that of $I$. Indeed, $\mu(I)-\mu(J)=\dim U_{0,1}$. We may choose the $y_i$ such that $y_1,\dots,y_t$ map to the minimal generators of $J$ and the remaining $y_{t+1},\dots,y_m$ map to $0$ (i.e. map to elements of $x_nV_1$ in the presentation of $\Rees(I)$). With these choices, $P_1+U=P(J)+(y_{t+1},\dots, y_m)$. By assumption, $P(I)$ has minimal generators only in degree $(1,1)$ and $(0,*)$. Then the same is true for $P_1$. Hence $P(J)+(y_{t+1},\dots, y_m)$ has minimal generators of degree $(1,1)$ and $(0,*)$. This is the desired result.

Finally we prove statement (3) by induction on the dimension.
Let $A=S[y_1,\dots,y_m]$ be the polynomial ring presenting $\Rees(I)$, and let $B=R[y_1,\dots,\allowbreak y_m]$ the corresponding polynomial ring for $\Rees(J)$.

Since $\mm_R$ annihilates $W$, we can see $W$ as a $K[y_1,\dots,y_m]$-module. Since $K[y_1,\dots,y_m]$ is an algebra retract of
$B$, we have
$$
P_B(W)=P_{K[y]}(W)P_{B}(K[y]);
$$
see Herzog \cite[Thm.1]{H} or Levin \cite[Thm.1.1]{L}. Therefore the $(1,0)$-regula\-ri\-ty of $W$ is $0$. By induction on the dimension we may assume that the $(1,0)$-regularity of $\Rees(J)$ is $0$, and we conclude that $\Rees(I)/x_n\Rees(I)$ has $(1,0)$-regularity $0$. This implies that the $(1,0)$-regularity of $\Rees(I)$ is $0$.
\end{proof}

So we have:
\begin{theorem}
\label{palombo}
The ideal $I$ has linear powers if and only if $\reg_{(1,0)} \Rees(I)=0$.
\end{theorem}

The implication $\Rightarrow$ has been proved in Lemma \ref{palo}. The converse has been proved by R\"omer \cite[Cor.5.5]{R}. An alternative proof is given in \cite[Thm.1.1]{HHZ}.

As the following example shows, ideals with linear powers need not to be of fiber type nor have a Koszul Rees algebra.
\begin{example}
\label{ex1}
The Rees algebra $\Rees(I)$ of the ideal
$$
I=(a^2b,a^2c, abd, b^2d)\subset K[a,b,c,d]
$$
is defined by the ideal
$$
P(I)=(-y_2b + y_1c, -y_3a + y_1d, -y_4a + y_3b, -y_3^2c + y_2y_4d)
$$
whose generators form a Gröbner basis. Hence the $(1,0)$-regularity of $\Rees(I)$ is $0$. It follows that $I$ has linear powers, it is not of fiber type and $\Rees(I)$ is not Koszul.
\end{example}

The following two examples are ideals with linear powers and of fiber type, but with a non-quadratic (hence non-Koszul) Rees algebra.

\begin{example}
\label{ex2}
A strongly stable ideal is a monomial ideal $I$ satisfying $I:x_i\supseteq I:x_j$ for every $i<j$. It is known that the regularity of a strongly stable ideal is equal to the largest degree of a minimal generator.

Let $I$ be a strongly stable ideal generated in degree $d$. The powers of $I$ are strongly stable as well, and hence $I$ has linear powers.
Moreover, strongly stable ideals are of fiber type; see Herzog, Hibi and Vladoiu \cite{HHV}. On the other hand, $\Rees(I)$ need not be quadratic. For example, the smallest strongly stable ideal of $K[x_1,x_2,x_3]$ containing the monomials $x_2^6, x_1^2x_2^2x_3^2,\allowbreak x_1^3x_3^3$ has a non-quadratic Rees algebra.
\end{example}

\begin{example}
\label{ex3}
Let $I$ be an ideal generated by monomials of degree $2$ and with a linear resolution. Then by \cite[Thm.2]{HHZ} $I$ has linear powers. Furthermore  Villarreal proved in  \cite[Thm.8.2.1]{V} that $I$ is of fiber type if it is square-free. On the other hand $\Rees(I)$ need not be quadratic. The ideal
$$
(x_3x_6, x_1x_3, x_5x_6, x_4x_6, x_2x_3, x_1x_5, x_3x_4, x_1x_6, x_1x_2, x_4x_5)
$$
has a linear resolution and its Rees algebra is not quadratic.
\end{example}

\section{Maximal minors with linear powers}
Let $X$ be a $m\times n$ matrix with entries in a Noetherian ring $R$ and let $I_j(X)$ denote the ideal of $R$ generated by the size $j$ minors of $X$. We will always assume that $m\leq n$. The ideal $I_m(X)$ of maximal minors  is resolved by the Eagon-Northcott complex provided $\grade(I_m(X))\geq n-m+1$.
In \cite[Thm.5.4]{ABW} Akin, Buchsbaum and Weyman for every $k\in \NN$ describe a complex of free $R$-modules resolving the $k$-th power of the ideal of maximal minors of $X$ under certain genericity conditions. In their notations the complex is denoted by $\YY(k,\phi)$ where $\phi$ is the $R$-linear map $R^n\otimes R^m\to R$ associated with $X$.  However we prefer to use the notation  $\YY(k,X)$ to stress the dependence on the input matrix $X$. They proved:

\begin{theorem}
\label{ABW}
Suppose $\grade I_j(X)\geq (m+1-j)(n-m)+1$ for every $j$. Then $I_m(X)^k$ is resolved by the free complex $\YY(k,X)$. Furthermore the length of $\YY(k,X)$ is $\min\{k,m\}(n-m)$.
\end{theorem}

\begin{remark}
The fact that the length of $\YY(k,X)$ is $\min\{k,m\}(n-m)$ is not explicitly stated in \cite{ABW}. Indeed in the proof of  \cite[Thm.5.4]{ABW} it is shown that the the length of $\YY(k,X)$ is $\leq \min\{k,m\}(n-m)$. 
To show that equality holds one can assume right away that $X$ is a matrix of variables $x_{ij}$ and that $R=K[x_{ij}]$ with $K$ a field of characteristic $0$. 
 Then,  according to \cite[Section 5]{ABW}, $\Tor_{i}^R(I_m(X)^k,K)$ is isomorphic to the kernel of a $\operatorname{GL}(V)\times \operatorname{GL}(W)$-equivariant homomorphism from 
$$
A=\bigwedge^{i}(V\otimes W)\otimes (L_{(m^k)}V\otimes L_{(m^k)}W)
$$ to 
$$
B=\bigwedge^{i-1}(V\otimes W)\otimes (L_{(m^k,1)}V\otimes L_{(m^k,1)}W),
$$
where $V$ and $W$ are $K$-vector spaces of dimension $m$ and $n$ and   $L_\lambda$ denotes the Schur functor associated with the partition $\lambda$.  The decompositions in irreducible $\operatorname{GL}(V)\times \operatorname{GL}(W)$-modules can be computed for all the involved representations, using the (skew) Cauchy formula and the Little\-wood-Richardson rule. In particular, when $i=\min\{k,m\}(n-m)$, one can check the following: 
\begin{itemize} 
\item[(1)] If $k\geq m$, then $L_{(m^{n+k-m})}V\otimes L_{(n^m,m^{k-m})}W$ is a direct summand of $A$ and not of $B$. 
\item[(2)] If $k< m$, then $L_{(m^k,k^{n-m})}V\otimes L_{(n^k)}W$ is a direct summand of $A$ and not of $B$. 
\end{itemize} 
So, in both cases, Schur's lemma implies that the kernel of the above map is not zero. It follows that $\Tor_{i}^R(I_m(X)^k,K)\neq 0$ for 
$i=\min\{k,m\}(n-m)$. Hence the projective dimension of   $I_m(X)^k$ is $\geq \min\{k,m\}(n-m)$.  \end{remark}

For later applications we record the following lemma and its corollary:

\begin{lemma}
\label{lemass}
Suppose that $\grade I_m(X)\ge n-m+1$ and $\grade I_j(X)\geq(m+1-j)(n-m)+2$ for every $j=1,\dots,m-1$, then
$$\Ass(R/I_m(X)^k)= \Ass(R/I_m(X))$$ for every $k>0$.
\end{lemma}

\begin{proof}
We use induction on $m$. Let $m=1$ and $P$ be an associated prime of
$I_m(X)^k$. After localization with respect to $P$ we may assume
that $R$ is local with maximal ideal $P$. The hypothesis implies
that $I=I_1(X)$ is now generated by a regular sequence. Then the
powers of $I$ are perfect ideals of grade $n$, resolved by the
Eagon-Northcott complex. Thus $\depth R_P=n$, and so $P$ is an
associated prime of $I$ as well.

Now let $m\ge 2$, and let $P$ be a prime ideal. Suppose first that
$P$ contains $I_1(X)$. We can apply Theorem \ref{ABW}. It shows that the projective dimension of $R/I_m(X)^k$ is at most
$m(n-m)+1$. The inequality for $\grade I_1(X)$ implies that $\depth
R_P>\projdim (R/I_m(X)^k)_P$. Thus $P$ is not associated to
$I_m(X)^k$, and it is not an associated prime of $I_m(X)$ for the
same reason.

Now suppose that $P$ does not contain $I_1(X)$. Then we can apply
the standard inversion argument to an entry of $X$, say $x_{11}$.
This argument reduces both $m$ and $n$ as well as the sizes of
minors by $1$, and since our bound on grade depends only on
differences of these numbers, they are preserved. Since the
inversion of an element outside $P$ does not affect the property of
$P$ being an associated prime, we can apply the induction
hypothesis.
\end{proof}

\begin{remark}
Lemma \ref{lemass} can also be derived from \cite[(9.27)(a)]{BV} that  gives a lower bound and the asymptotic value for the depth of   
$R/I_m(X)^k$.
\end{remark}

\begin{corollary}
\label{asspro}
Let $X$ be a $m\times n$ matrix with entries in a Noetherian ring
$R$. Suppose that for some number $p$ with $1\leq p\leq m$, we have:
\begin{itemize}
\item[(i)] $\grade I_m(X)\ge n-m+1$,
\item[(ii)] $\grade I_j(X)\geq (m+1-j)(n-m)+2$ for all $j=p+1,\dots,m-1$.
\end{itemize}
Then
$$\Ass(R/I_m(X)^k)\subseteq \Ass(R/I_m(X)) \cup \{ P : P\supseteq I_p(X)\}$$
for every $k>0$.
\end{corollary}

\begin{proof}
Again we use induction on $m$. If a prime ideal $P$ contains
$I_p(X)$, there is nothing to prove for $P$. Otherwise it does not
contain a $p$-minor. Its inversion reduces all sizes by $p$. But
then we are in the situation of Lemma \ref{lemass}.
\end{proof}

From now one we assume that $S$ is a polynomial ring over a field $K$ and $X$ is a $m\times n$ matrix whose entries are linear forms in $S$.
In this setting we have:

\begin{proposition}
\label{ABWEH}
Let $X$ be a $m\times n$ matrix of linear forms with $m\leq n$. Suppose $\height I_j(X)\geq (m+1-j)(n-m)+1$ for every $j$. Then
\begin{itemize}
\item[(1)] $I_m(X)$ has linear powers.
\item[(2)] The Rees algebra $\Rees(I_m(X))$ is Koszul.
\end{itemize}
\end{proposition}
\begin{proof} The first assertion is a special case of  Theorem \ref{ABW} because the complexes $\YY(k,X)$ are linear when the entries of $X$ are linear.
The second assertion follows from the result of Eisenbud and Huneke \cite[Thm.3.5]{EHu} since, under the assumption of Proposition \ref{ABWEH}, the Rees algebra of $I_m(X)$ is a quotient of the generic one by a regular sequence of linear forms. (The generic one is the Rees algebra of the ideal of maximal minors of a $m\times n$ matrix of distinct variables over the base field). Hence it is enough to prove the statement for a matrix of variables. The Rees algebra of the ideal of minors of a matrix of variables is a homogeneous ASL \cite[Thm.9.14]{BV} and hence it is defined by a Gröbner basis of quadrics. Therefore it is Koszul.
\end{proof}

Our main result is the following:

\begin{theorem}
\label{mainThm}
Let $X$ be a $m\times n$ matrix with $m\leq n$ whose entries are linear forms in a polynomial ring $S$ over a field $K$.
Suppose that $\height I_m(X) \geq n-m+1$ and $\height I_j(X)\geq \min\{ (m+1-j)(n-m)+1, N \}$ for every $j=2,\dots,m-1$ where $N=\height I_1(X)$. Then:
\begin{itemize}
\item[(1)] $I_m(X)$ has linear powers and it is of fiber type.
\item[(2)] The Betti numbers of $I_m(X)^k$ depends only on $m,n,k$ and $N$.
\end{itemize}
\end{theorem}

To prove Theorem \ref{mainThm} we need a sort of deformation argument.

\begin{lemma}
\label{yA}
Let $X$ be a $m\times n$ matrix whose entries are linear forms in a polynomial ring $S$ over an infinite field $K$. Let $y$ be a new variable. Then:
\begin{itemize}
\item[(1)] For every $A\in M_{mn}(K)$ and for every $j$ one has
$$ \height I_j(X)\leq \height I_j(X+yA)\leq 1+\height I_j(X).$$
\item[(2)] There exists $A\in M_{mn}(K)$ such that $$\height I_j(X+yA)= \min\{ (m+1-j)(n+1-j), \height I_j(X)+1\}$$ for every $j$.
\end{itemize}
\end{lemma}

\begin{proof} (1) The inequality on the right follows from the inclusion $I_j(X+yA)\subseteq (y)+I_j(X)$. For the other we consider the weight vector $w$ that gives weight $1$ to the variables of $S$ and weight $0$ to $y$. One has $\ini_w(I_j(X+yA))\supseteq I_j(X)$. Since the height does not change by taking ideals of the initial forms, we conclude that $\height I_j(X)\leq \height I_j(X+yA)$.

(2) The inequality $\leq$ follows from (1). Set $c=\dim S$, $h_j=\height I_j(X)$, $g_j=(m+1-j)(n+1-j)$ and $T=\{ j : h_j<g_j\}$. By virtue of (1) and since $\height I_j(Y)\leq g_j$ holds for every $m\times n$ matrix $Y$, it is enough to show that there exists $A\in M_{mn}(K)$ such that $\height I_j(X+yA)=h_j+1$ for every $j\in T$.
We choose a system of coordinates $x_1,\dots,x_c$ for $S$ such that for every $j$ the set
$E_j=\{ x_k : k>h_j\}$ is a system of parameters for the ring $S/I_j(X)$. This can be done because the base field $K$ is infinite. It follows that $I_j(X)$ has height $h_j$ also modulo $(E_j)$. We may write $X=\sum_{i=1}^c x_i A_i$ with $A_i\in M_{mn}(K)$.
Since $X=\sum_{i=1}^{h_j} x_i A_i$ mod $(E_j)$, we deduce that $\height I_j(\sum_{i=1}^{h_j} x_i A_i)=h_j$, that is, the radical of $I_j(\sum_{i=1}^{h_j} x_i A_i)$ is $(x_1,\dots,x_{h_j})$, for every $j$.

Again by virtue of (1), it is enough to show that there exists $A\in M_{mn}(K)$ such that $\height I_j(\sum_{i=1}^{h_j}x_iA_i+yA)=h_j+1$ for every $j\in T$.
We consider the subvariety $V_j$ of the projective space $\PP(M_{mn}(K))$ of the matrices of rank $<j$. Furthermore we consider the linear space $L_j=\langle A_i : i=1,\dots, h_j\rangle\subseteq \PP(M_{mn}(K))$. By construction $L_j\cap V_j=\emptyset$. Consider then the join $L_j*V_j$ of $L_j$ and $V_j$, that is $L_j*V_j=\cup \overline{BC}$ where $\overline{BC}$ is the line joining $B$ and $C$ and $(B,C)$ varies in $L_j\times V_j$. It is well know that $L_j*V_j$ is a projective variety and that
$$\dim L_j*V_j= \dim L_j+ \dim V_j+1=mn-1-(g_j-h_j),$$
see for instance \cite[11.37]{Ha}.
Therefore $L_j*V_j$ is a proper subvariety of $\PP(M_{mn}(K))$ if $j\in T$, and so is $\bigcup_{j\in T} ( L_j*V_j )$. It follows that we can take $A\in M_{mn}(K)$ and $A\not\in L_j*V_j$ for every $j\in T$. We claim that with this choice of $A$ one has
$\height I_j(\sum_{i=1}^{h_j}x_iA_i+yA)=h_j+1$ for every $j\in T$ as desired. Suppose, by contradiction, that for a $j\in T$ one has $\height I_j(\sum_{i=1}^{h_j}x_iA_i+yA)<h_j+1$. Then there is a point $(a_1,\dots,a_{h_j},b)$ on the projective subvariety of $\PP^{h_j}$ defined by the ideal $I_j(\sum_{i=1}^{h_j}x_iA_i+yA)$ (where the last coordinate correspond to the variable $y$). In other words $\sum_{i=1}^{h_j}a_iA_i+bA$ has rank $<j$. If $b\neq 0$ we obtain that $A\in L_j*V_j$, a contradiction. If $b=0$ we have $(a_1,\dots,a_{h_j})\in L_j\cap V_j$ which is also a contradiction.
 \end{proof}

\begin{proof} [Proof of Theorem \ref{mainThm}]
(1) It is harmless to assume that the base field is infinite. Set $N= \height I_1(X)$. We may assume that $N=\dim S$. We have $N\leq mn$ and we do decrising induction on $N$. If $N=mn$ then we are in the generic case and the assertion is true because of Proposition \ref{ABWEH}.
Consider a new variable $y$, and set $R=S[y]$. In view of Lemma \ref{yA} we may take $A\in M_{mn}(K)$ such that the ideals of minors of the the matrix $Y=X+yA$ satisfy
\begin{equation}\label{eq1}
\height I_j(Y)=\min\{(m+1-j)(n+1-j), \height I_j(X)+1\}
\end{equation}
for every $j$. Hence $\height I_1(Y)=N+1$ and since, by assumption, we have $\height I_j(X)\geq \min\{(m+1-j)(n-m)+1, N\}$ we may deduce that
\begin{equation}\label{eq2}
\height I_j(Y)\geq \min\{(m+1-j)(n+1-j), (m+1-j)(n-m)+2, N+1 \}
\end{equation}
In particular,
\begin{equation}\label{eq3}
\height I_j(Y)\geq \min\{ (m+1-j)(n-m)+1, N+1 \}
\end{equation}
holds for every $j$. Hence, by induction, we may assume that $I_m(Y)$ has linear powers and is of fiber type. Since $R/I_m(Y)\otimes S=S/I_m(X)$ we conclude from Lemma \ref{reg} that $I_m(X)$ has linear powers and is of fiber type provided we show that
\begin{equation}\label{eq4}
(I_m(Y)^k:y)/I_m(Y)^k
\end{equation}
has finite length. Now \eqref{eq2} implies that
\begin{equation}\label{eq5}
\height I_j(Y)\geq \min\{ (m+1-j)(n-m)+2, N+1 \}
\end{equation}
for every $j\leq m-1$. Let $p$ be the largest number such that $(m+1-p)(n-m)+2>N+1$. Hence $\height I_j(Y)\geq (m+1-j)(n-m)+2$ for $j=p+1,\dots,m-1$. From Corollary \ref{asspro} we deduce that the associated primes of $I_m(Y)^k$ are either the minimal primes of $I_m(Y)$ or ideals containing $I_p(Y)$. But, by construction, $\height I_p(Y)=N+1$ and hence the radical of $I_p(Y)$ is the maximal homogeneous ideal $\mm_R$ of $R$. Summing up,
$$
\Ass(R/I_m(Y)^k)\subset \Min(I_m(Y)) \cup \{\mm_R\}.
$$

Since $\height I_m(Y)=\height I_m(X)$, the $1$-form $y$ does not belong to the minimal primes of $\height I_m(Y)$. We may hence conclude that $I_m(Y)^k:y$ is contained in the saturation $I_m(Y)^k:\mm_R^{\infty}$ of $I_m(Y)^k$ and that the module \eqref{eq4} has finite length.

Finally (2) follows from the construction above and Lemma \ref{palo}(2).
\end{proof}

As a special case of Theorem \ref{mainThm} we have:

\begin{corollary}
\label{ans}
Let $X$ be a $2\times n$ matrix with linear entries in a polynomial ring $S$. Assume $I_2(X)$ has height $n-1$. Then $I_2(X)$ has linear powers. In particular, any rational normal scroll has linear powers.
\end{corollary}

The following examples show that ideals of maximal minors with the expected codimension need not to have linear powers.

\begin{example} The ideal of $3$-minors of the following $3\times 5$ matrix has height $3$ and its square does not have a linear resolution.
$$X=\left(
\begin{array}{ccccc}
 x_1 & 0 &   0 & x_2 & x_4 \\
 0 &   0 & x_3 & x_2 & x_5 \\
 0 & x_2 & x_1 & x_3 & x_3
\end{array}
\right)
$$
The ideal of $2$-minors has only height $3$.
\end{example}

\begin{example} The ideal of $4$-minors of the following $4\times 5$ matrix has height $2$, and its square does not have a linear resolution.
$$X=\left(
\begin{array}{ccccc}
x_1&0&0&0&x_3 \\
0&x_2&0&0&x_4 \\
0&0& x_2&x_3&0 \\
0&0& x_1&x_4&x_3
\end{array}
\right)
$$
The ideal of $3$-minors has only height $2$.
 \end{example}

 On the other hand, the ideal of Example \ref{ex1} is the ideal of $3$-minors of a $3\times 4$ matrix, it has linear powers and it is not of fiber type. So it is an ideal of maximal minors with expected codimension that is not of the type described in Theorem \ref{mainThm}, but nevertheless has linear powers.

The irreducible varieties of minimal degree are the rational normal scrolls, the quadrics hypersurfaces, and the Veronese surface in $\PP^5$, \cite{EHa}. We have proved above that the rational normal scrolls have linear powers and for the quadric hypersurfaces that is obvious. A computer assisted proof that the Veronese surface in $\PP^5$ has linear powers is given in \cite{HHO}. We give below an alternative proof.

\begin{proposition}
\label{Ver}
The Veronese surface in $\PP^5$ has linear powers.
\end{proposition}

\begin{proof} The Veronese surface in $\PP^5$ is defined by the ideal $I\subset S=K[x_0,\dots,\allowbreak x_5]$ of the $2$-minors of the generic symmetric $3\times 3$ matrix.
Its general hyperplane section is the rational normal curve of $\PP^4$.
Hence if $y$ is a general linear form, $S/I^k\otimes S/(y)=R/J^k$ where $J$ in $R=K[x_0,\dots,x_4]$ defines the rational normal curve of $\PP^4$. So we know that $\reg(R/J^k)=2k-1$. And it remains to control $H_{\mm}^0(S/I^k)$.
But we know the primary decomposition of $I^k$,
$$I^k=I^{(k)}\cap \mm^{2k}.$$
This is proved by Abeasis in \cite[Thm.5.1,Cor.5.2]{A} in characteristic $0$. Using the ideas developed in \cite[Chap.10]{BV} and, in particular, \cite[Lem.10.10]{BV} one shows that the same result holds in arbitrary characteristics.
This implies that $H_{\mm}^0(S/I^k)$ vanishes in degrees $\geq 2h$. From Lemma \ref{reg} it then follow that $\reg(I^k)=2k$.
\end{proof}

The ideal $I_{n-1}(Y_n)$ of $(n-1)$-minors of a symmetric $n\times n$ matrix of variables $Y_n$ has a linear resolution. We have seen in Proposition\ref{Ver} that $I_2(Y_3)$ has linear powers. As observed in \cite[Ex.2.8]{C}, $I_3(Y_4)$ does not have linear powers because the resolution of $I_3(Y_4)^2$ is not linear. One can also check that $I_3(Y_4)^3$ has a linear resolution and so perhaps $I_3(Y_4)$ has linear powers with the exception of the second one, but this is another story.

\end{document}